\documentclass[twocolumn,noversion,nonote,squeezed]{cdcarticle}
\usepackage{graphicx}
\graphicspath{{graphics/}}
\usepackage{cite}
\usepackage{color}
\usepackage{bm}
\usepackage[hidelinks]{hyperref}

\usepackage[font=footnotesize,labelfont=bf,margin=1em]{caption}
\usepackage{math}

\usepackage{tikz}
\usetikzlibrary{positioning,calc}

\usepackage{booktabs}
\usepackage[shortlabels]{enumitem}

\usepackage{stfloats}  

\let\OLDthebibliography\thebibliography
\renewcommand\thebibliography[1]{
  \OLDthebibliography{#1}
  \setlength{\parskip}{0pt}
  \setlength{\itemsep}{0pt plus 0.3ex}
}

\theoremstyle{definition}

\newcommand{\squeezemat}[1]{\addtolength{\arraycolsep}{#1}}

\newcommand{\FmL}{\mathcal{F}_{m,L}}
\newcommand{\f}{\bm{f}}
\renewcommand{\u}{\bm{u}}
\newcommand{\x}{\bm{x}}
\newcommand{\y}{\bm{y}}
\newcommand{\e}{\bm{e}}
\newcommand{\A}{\bm{A}}
\newcommand{\B}{\bm{B}}
\renewcommand{\C}{\bm{C}}
\renewcommand{\D}{\bm{D}}
\renewcommand{\H}{\bm{H}}
\newcommand{\X}{\bm{X}}
\newcommand{\Y}{\bm{Y}}
\newcommand{\U}{\bm{U}}

\DeclareMathOperator{\rate}{\textsc{rate}}
\DeclareMathOperator{\sensitivity}{\textsc{sens}}

\title{A Tutorial on a Lyapunov-Based Approach to the\\Analysis of Iterative Optimization Algorithms}
\author{Bryan Van Scoy\thanks{B. Van Scoy is with the Department of Electrical and Computer Engineering at Miami University, Oxford, OH 45056, USA
\texttt{bvanscoy@miamioh.edu}}\and Laurent Lessard\thanks{L. Lessard is with the Department of Mechanical and Industrial Engineering at Northeastern University, Boston, MA 02115, USA \texttt{l.lessard@northeastern.edu\vspace{2pt}}}}
\note{IEEE Conference on Decision and Control, 2023}

\begin{document}
\maketitle

\begin{abstract}
Iterative gradient-based optimization algorithms are widely used to solve difficult or large-scale optimization problems. There are many algorithms to choose from, such as gradient descent and its accelerated variants such as Polyak's Heavy Ball method or Nesterov's Fast Gradient method. It has long been observed that iterative algorithms can be viewed as dynamical systems, and more recently, as robust controllers. Here, the ``uncertainty'' in the dynamics is the gradient of the function being optimized.
Therefore, worst-case or average-case performance can be analyzed using tools from robust control theory, such as integral quadratic constraints (IQCs). In this tutorial paper, we show how such an analysis can be carried out using an alternative Lyapunov-based approach. This approach recovers the same performance bounds as with IQCs, but with the added benefit of constructing a Lyapunov function.

\end{abstract}

{\let\thefootnote\relax\footnote{This material is based upon work supported by the National Science Foundation under Grants No. 2136945, 2139482.}}%

\section{Introduction}

In this paper, we consider unconstrained optimization problems of the form
$
    \min_{x\in \R^d}\, f(x),
$
where $f:\R^d\to\R$ is a continuously differentiable function. 

Iterative gradient-based optimization algorithms attempt to solve such problems by starting from some initial guess $x_0\in\R^d$ and iteratively updating $x_k$ in an effort to converge to a local minimizer $x_\star \in \argmin_x\, f(x)$. The simplest such algorithm is \emph{gradient descent} \eqref{GD}, which uses the update rule
\begin{equation}\label{GD}
    x_{k+1} = x_k - \alpha \grad f(x_k). \tag{GD}
\end{equation}
Intuitively, each step moves in the direction of the negative gradient (steepest descent direction). Here, $\alpha$ is the \emph{stepsize}, which is a tunable parameter. Generally, a larger stepsize can result in faster convergence, but if the stepsize is too large, the algorithm may fail to converge.

Gradient descent may converge slowly when the function is poorly conditioned (when the condition number of the Hessian $\nabla^2 f$ is large). This is due to the fact that the contours of $f$ are elongated so the iterates tend to oscillate in a non-productive manner. One way to alleviate this problem is by using \emph{accelerated algorithms}. For example, Polyak's Heavy Ball \eqref{HB} \cite[\S3.2.1]{polyak_book} uses an additional \emph{momentum term} compared to \eqref{GD}:
\begin{equation}\label{HB}
    x_{k+1} = x_k -\alpha \grad f(x_k) + \beta (x_k - x_{k-1}). \tag{HB}
\end{equation}
Alternatively, Nesterov's Fast Gradient \eqref{FG} \cite[\S2.2.1]{nesterov_book} evaluates the gradient at an interpolated point $y_k$:
\begin{equation}\label{FG}\tag{FG}
\begin{aligned}
    x_{k+1} &= y_k -\alpha \grad f(y_k), \\
    y_{k+1} &= x_{k+1} + \beta(x_{k+1} - x_{k}).
\end{aligned}
\end{equation}
In \cref{fig:convergence}, we compare the convergence of \eqref{GD}, \eqref{HB}, and \eqref{FG} on a simple quadratic function.

\begin{figure}[ht]
    \centering
    \includegraphics{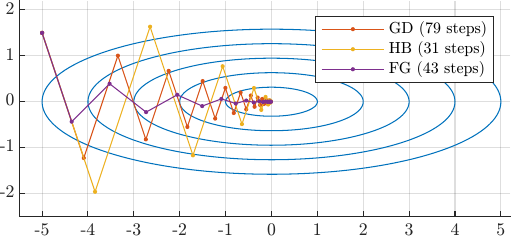}
    \caption{Comparison of different iterative algorithms applied to $f(u,v) = u^2 + 10 v^2$ (contour lines shown) with initial point $(-5,1.5)$. Each algorithm is tuned to have the fastest possible convergence rate. Steps (in brackets) indicate the number of iterations needed to achieve $\norm{x_k - x_\star} < 10^{-6}$. Gradient descent (GD) is less efficient than the Heavy Ball (HB) or Fast Gradient (FG) accelerated methods.}
    \label{fig:convergence}
\end{figure}

The tuning of an algorithm (the choice of $\alpha$ and $\beta$ in the examples above) can have a dramatic effect on the convergence behavior. However, it is typically not feasible to find the optimal tuning, as this would depend on $f$, which presumably is a function complicated enough to warrant being minimized numerically.
Instead, we seek performance guarantees over a \emph{class of functions}. For example, we may know a priori that $f$ is convex, or that it possesses a certain structural property, such as being quadratic. In the optimization literature, this sort of worst-case analysis is known as \emph{algorithm analysis}.

\subsection{The quadratic case}

When $f$ is a quadratic function as in least squares problems and as illustrated in \cref{fig:convergence}, we can explicitly parameterize $f$, which makes the algorithm analysis straightforward. Specifically, $\grad f(x) = Qx$ for some matrix $Q$. The dynamical systems for GD, HB, FG are therefore linear and convergence rate can be established via eigenvalue analysis. This leads to the following result.
\begin{prop}[\!\!{\cite[Prop.~1]{lessard16}}] \label{prop:quadratic}
    Suppose $f:\R^d\to\R$ is quadratic and satisfies $0 \prec m I_d \preceq \nabla^2 f(x) \preceq L I_d$. The smallest $\rho$ such that the iterates of the algorithms GD, HB, FG applied to $f$ satisfy
    \[
    \norm{x_k - x_\star} \leq C \rho^k \norm{x_0-x_\star}
    \quad\text{for some $C > 0$}
    \]
    is given by the following tunings (where $\kappa \defeq L/m$)

    \begin{table}[ht]
        \renewcommand{\arraystretch}{1.3}
        \centering
        \begin{tabular}{lll}
            \toprule
             & Optimal tuning & Rate bound \\
            \midrule
            \eqref{GD} & $\alpha = \frac{2}{L+m}$ & $\rho = \frac{\kappa-1}{\kappa+1}$\\
            \eqref{HB} & $\alpha = \frac{4}{(\sqrt L + \sqrt m)^2},\, \beta = \bigl(\frac{\sqrt{\kappa}-1}{\sqrt{\kappa}+1}\bigr)^2$ & $\rho = \frac{\sqrt{\kappa}-1}{\sqrt{\kappa}+1}$\\
            \eqref{FG} & $\alpha = \frac{4}{3L+m},\, \beta = \frac{\sqrt{3\kappa+1}-2}{\sqrt{3\kappa+1}+2}$ & $\rho = \frac{\sqrt{3\kappa+1}-2}{\sqrt{3\kappa+1}}$\\
            \bottomrule
        \end{tabular}
    \end{table}
\end{prop}

The tunings from \cref{prop:quadratic} are the same as those used to generate the iterates shown in \cref{fig:convergence}.

\subsection{The smooth and strongly convex case}\label{sec:FmL}

Another popular function class is the set of \emph{smooth and strongly convex functions}, which we denote $\FmL$. These are continuously differentiable functions $f$ that satisfy:
\begin{enumerate}[(i)]
    \item $L$-Lipschitz gradients: $\norm{\grad f(x)\!-\!\grad f(y)} \!\leq\! L \norm{x\!-\!y}$ for all $x,y\in\R^d$.
    \item $m$-strong convexity: $f(x) - \frac{m}{2}\norm{x}^2$ is convex.
\end{enumerate}
Here, we assume $m$ and $L$ can be estimated or are otherwise known a priori. This class of functions includes the quadratic functions from \cref{prop:quadratic}, but also includes non-quadratic functions. Functions of this type occur for example in regularized logistic regression or support vector machines with a smoothed hinge loss \cite{statistical_learning_book}.

The class of smooth strongly convex functions cannot easily be parameterized, and the closed-loop dynamics of algorithms such as GD, HB, FG are generally nonlinear. Therefore, performing algorithm analysis requires a different approach from the quadratic case. Classical approaches typically involve clever manipulation of the inequalities that characterize Lipschitz gradients or strong convexity. One such approach is Nesterov's \emph{estimating sequences} \cite[\S2.2.1]{nesterov_book}, which yield the following result.

\begin{prop}\label{prop:estimating_sequences}
    Suppose $f:\R^d\to\R$ is smooth and strongly convex with parameters $0<m\leq L$. The method \eqref{FG} with parameters $\alpha = \frac{1}{L}$ and $\beta = \frac{\sqrt{L}-\sqrt{m}}{\sqrt{L}+\sqrt{m}}$ achieves 
    $
    \norm{x_k - x_\star} \leq C \rho^k \norm{x_0-x_\star}
    $
    for some $C>0$, with $\rho = \sqrt{1-\sqrt{\tfrac{m}{L}}}$.
\end{prop}

\subsection{Automated algorithm analysis}

In recent years, there has been an effort to systematize, unify, and automate the process of algorithm analysis.

Drori and Teboulle \cite{drori-teboulle} pioneered the idea that the worst-case performance of an algorithm (after a fixed number of iterations) could be characterized by a semidefinite program (SDP) whose size scaled with the number of iterations. It was later shown that this SDP could be convexified, leading to the so-called \emph{Performance Estimation Problem} formulation \cite{taylor2017smooth}.

Another popular approach is to express iterative methods as discretizations of gradient flows \cite{su2014differential}. This view led to the unification of various accelerated algorithms, including \eqref{FG}, and the derivation of an associated Lyapunov-based proof of convergence \cite{wilson2021lyapunov}.

Yet another approach is to look for an asymptotic performance guarantee by leveraging tools from robust control \cite{lessard16}. Here, the idea is to view the algorithm as a Lur'e problem \cite{lure_postnikov} (a linear time-invariant system in feedback with a static nonlinearity) and to adapt the integral quadratic constraint (IQCs) approach \cite{iqc} to derive performance bounds. This approach also requires solving a convex SDP, but unlike PEP, its size is small and fixed (does not depend on the number of iterations).

\subsection{Tutorial overview}

In this tutorial, we will present an alternative Lyapunov-based approach to algorithm analysis, described in more detail in \cite{optalg,dissalg,van2022absolute}, that shares features in common with each of the aforementioned works.
In the sections that follow, we will: 1) frame algorithm analysis as a Lur'e problem (as in the IQC approach); 2) use interpolation conditions to describe the set of smooth strongly convex functions (as in the PEP approach); 3) describe the lifting procedure we will use; 4) show how Lyapunov functions can be used to certify performance via solving convex SDPs; and 5) present numerical examples that illustrate the effectiveness of this approach for certifying both convergence rate and robustness to gradient noise.

\section{Iterative algorithms as Lur'e problems}

To express the iterative algorithm as a Lur'e problem, we begin by defining signals corresponding to the input $y_k$ and output $u_k$ of the gradient $\grad f$. For example, \eqref{FG} can be rewritten as
\begin{align*}
    x_{k+1} &= y_k - \alpha u_k, \\
    y_{k+1} &= x_{k+1} + \beta (x_{k+1} - x_k), \\
    u_k &= \grad f(y_k).
\end{align*}
Now define the augmented state $\xi_k \defeq (x_k,x_{k-1})$ and put into standard state-space form to obtain
\begin{subequations}\label{eq:ss_nest}
\begin{align}
    \xi_{k+1} &= \bmat{ (1+\beta)I_d & -\beta I_d \\ I_d & 0_d} \xi_k + \bmat{-\alpha I_d \\ 0_d} u_k, \\
    y_k &= \bmat{ (1+\beta)I_d & -\beta I_d } \xi_k, \\
    u_k &= \grad f(y_k).
\end{align}
\end{subequations}
To avoid the use of Kronecker products, it is convenient to express states, inputs, and outputs as \emph{row vectors}, where there are $d$ columns corresponding to the $d$ dimensions of the original system. This transforms \eqref{eq:ss_nest} into%
\begin{subequations}\label{eq:ss_nest2}
    \begin{align}
        \xi_{k+1} &= \bmat{ 1+\beta & -\beta \\ 1& 0} \xi_k + \bmat{-\alpha  \\ 0} u_k, \\
        y_k &= \bmat{ 1+\beta & -\beta } \xi_k, \\
        u_k &= \grad f(y_k),
    \end{align}
    \end{subequations}
where $\xi_k \in \R^{2\times d}$, $u_k \in \R^{1\times d}$, $y_k \in \R^{1\times d}$, and the gradient maps row vectors to row vectors: $\grad f : \R^{1\times d} \to \R^{1\times d}$.
Because of this convention, $\norm{\cdot}$ denotes the Frobenius norm, for example, $\norm{\xi_k}^2 \defeq \sum_{i=1}^2\sum_{j=1}^d (\xi_k)_{ij}^2$.

We will use this convention from now on. Viewed as a block diagram, \cref{eq:ss_nest2} has the following representation.%
\begin{figure}[ht]
    \centering
    \begin{tikzpicture}[thick]
        \tikzstyle{block}=[draw,rounded corners,minimum width=1cm, minimum height=8mm,inner sep=2mm]
        \tikzstyle{arr}=[->,>=latex,rounded corners]
        \node[block] (G) {$G$};
        \node[block, below= 4mm of G] (df) {$\grad f$};
        \draw[arr] (G) -- +(1.2,0) |- node[pos=0.25,auto]{$y$} (df);
        \draw[arr] (df) -- +(-1.2,0) |- node[pos=0.25,auto]{$u$} (G);
    \end{tikzpicture}
\end{figure}

\noindent Here, $G$ is a single-input single-output (SISO) system that depends on the algorithm, and it is understood that $G$ acts separately on each of the $d$ dimensions of the input. The system $G$ has different realizations depending on the algorithm:
\begin{table}[ht]
    \centering
    \begin{tabular}{ccc}
    \toprule
        GD & HB & FG \\
    \midrule
        $\!\left[\begin{array}{c|c}
            1 & -\alpha \\\hline
            1 & 0
        \end{array}\right]\!$ &
        $\squeezemat{-2pt}\!\left[\begin{array}{cc|c}
            1+\beta & -\beta & -\alpha \\
            1 & 0 & 0 \\ \hline
            1 & 0 & 0
        \end{array}\right]\!$ &
        $\squeezemat{-2pt}\!\left[\begin{array}{cc|c}
            1+\beta & -\beta & -\alpha \\
            1 & 0 & 0 \\ \hline
            1+\beta & -\beta & 0
        \end{array}\right]\!$ \\
    \bottomrule
    \end{tabular}
\end{table}

The idea of representing iterative algorithms as dynamical systems is not new; see for example \cite{polyak_book}. It is not restricted to the three algorithms above either; it can be used to represent distributed optimization algorithms \cite{distralg}, operator-splitting methods \cite{iqcadmm_ICML,dissalg}, and a variety of other numerical algorithms \cite{bhaya2006control}.

\section{Lifted algorithm dynamics}

Suppose the algorithm dynamics satisfy the state space equations (again using the row vector convention)
\begin{subequations}\label{eq:ssalg}
    \begin{align}
        \xi_{k+1} &= A \xi_k + B \left( u_k + w_k \right), \\
        y_k &= C \xi_k, \\
        u_k &= \grad f( y_k ).
    \end{align}
\end{subequations}
In contrast with \eqref{eq:ss_nest2}, we have included gradient noise $w_k$, which will allow us to analyze algorithm sensitivity in the sequel.
In order to find the tightest possible performance bounds, we will search for a Lyapunov function that depends on a finite history of past algorithm iterates and function values. To this end, define the \emph{lifted} iterates
\begin{equation}\label{eq:lifted_iterates}
    \y_k \defeq \bmat{y_k\\y_{k-1}\\\vdots\\y_{k-\ell}}\!,\;
    \u_k \defeq \bmat{u_k\\u_{k-1}\\\vdots\\u_{k-\ell}}\!,\;
    \f_k \defeq \bmat{f_k\\f_{k-1}\\\vdots\\f_{k-\ell}}\!.
\end{equation}
Also, define the truncation matrices $Z_+$, $Z$ as follows.
\begin{equation}\label{eq:Z}
    Z_+ \defeq \bmat{I_\ell & 0_{\ell \times 1}}
    \quad\text{and}\quad
    Z \defeq \bmat{0_{\ell \times 1} & I_\ell}.
\end{equation}
Multiplying a lifted iterate by $Z$ on the left removes the most recent iterate (time $k$), while using $Z_+$ removes the oldest iterate (time $k-\ell$). Now define the lifted state
\begin{equation}\label{eq:state_aug}
	\x_k \defeq \bmat{ \xi_k-\xi_\star \\ Z \y_k \\ Z \u_k} \in \R^{(n+2\ell)\times d},
\end{equation}
which is the current state $\xi_k$ and the $\ell$ previous inputs $u_{k-1},\ldots,u_{k-\ell}$ and outputs $y_{k-1},\ldots,y_{k-\ell}$ of the original system. The \emph{lifted system} is a system with state $\x_k$, input $(u_k,w_k)$, and output $(\y_k,\u_k)$ that is consistent with the original dynamics \eqref{eq:ssalg}. These turn out to be
  \begin{align}
    \x_{k+1}\! &= \!\underbrace{\squeezemat{-3pt}
    \bmat{ A & 0 & 0 \\
    Z_+ \e_1 C & Z_+ Z^\tp & 0 \\
    0 & 0 & Z_+ Z^\tp }}_{\A}\!\!\x_k + \! \underbrace{\bmat{B \\ 0 \\ Z_+ \e_1\!}}_{\B}\!\! u_k +\! \underbrace{\bmat{B \\ 0 \\ 0}}_{\H}\!\! w_k \notag \\
    \bmat{\y_k \\ \u_k}\! &= \!\underbrace{\bmat{ \e_1 C & Z^\tp & 0 \\ 0 & 0 & Z^\tp }}_{\C} \x_k + \underbrace{\bmat{ 0 \\ \e_1}}_{\D} u_k, \label{eq:sslifted}
  \end{align}
where $\e_1=\bmat{1 & 0 & \cdots & 0}^\tp \in\R^{\ell+1}$. We can recover the iterates of the original system (shifted to the fixed point) by projecting the augmented state and the input as
\begin{multline}
	\tilde \xi_k = \underbrace{\bmat{I_n & 0_{n\times(2\ell+1)}}}_{\X} \!\bmat{\x_k \\ u_k}\!, \;
	\tilde y_k = \underbrace{\bmat{C & 0_{1\times(2\ell+1)}}}_{\Y} \!\bmat{\x_k \\ u_k}\!, \\
  \text{and}\quad
	\tilde u_k = \underbrace{\bmat{0_{1\times(n+2\ell)} & 1}}_{\U} \bmat{\x_k \\ u_k}. \label{eq:sslifted2}
\end{multline}
Note that when $\ell=0$, the lifted system reduces to the original system. That is, \eqref{eq:sslifted} becomes \eqref{eq:ssalg}.

\section{Interpolation Conditions}\label{sec:interpolation}

For the remainder of this paper, we will restrict our attention to smooth strongly convex functions, $f\in\FmL$.

The characterization of $\FmL$ from \cref{sec:FmL} depends on a continuum of points $x,y\in\R^d$. However, we will be analyzing the algorithm at a discrete set of iterates, so we will instead characterize $\FmL$ using \emph{interpolation conditions}. These conditions were developed in \cite{taylor2017smooth} and provide necessary and sufficient conditions under which a given finite set of input-output data can be interpolated by a function in $f\in\FmL$.
\begin{thm}[{\!\!\cite[Thm.~4]{taylor2017smooth}}]\label{thm:interp}
    Consider the set $\{(y_k,u_k,f_k)\}$ for $k=1,\dots,m$. The following are equivalent.
    \begin{enumerate}[(i)]
        \item There exists a function $f\in\FmL$ satisfying
        \[
          f(y_k) = f_k\quad\text{and}\quad \grad f(y_k) = u_k\quad\text{for }k=1,\dots,m.  
        \]
        \item The following inequality holds for $i,j \in \{1,\dots,m\}$.\label{interpii}
        \begin{multline*}
            2(L-m)(f_i-f_j)-mL\norm{y_i-y_j}^2\\
            +2(y_i-y_j)^\tp(mu_i-Lu_j)-\norm{u_i-u_j}^2\geq 0.
        \end{multline*}
    \end{enumerate}
\end{thm}

\noindent Although we restrict our attention to smooth strongly convex functions, interpolation conditions can be derived for other function classes as well. See for example \cite{taylor2017smooth,bousselmi2023interpolation}.

Using our row vector convention, if we suppose that $y_k,u_k\in\R^{1\times d}$, \cref{interpii} in \cref{thm:interp} can be written as
\begin{equation}\label{eq:qij}
  q_{ij} \defeq
  \trace\bmat{y_i\\y_j\\u_i\\u_j}^\tp \!\!\!H \bmat{y_i\\y_j\\u_i\\u_j} + h^\tp \bmat{f_i \\ f_j} \geq 0,\quad\text{where}    
\end{equation}
\[
H \defeq
\squeezemat{-3pt}
\bmat{ -mL & mL & m & -L\\ mL & -mL & -m & L\\ m & -m & -1 & 1\\ -L & L & 1 & -1}
\;\text{and}\;\;
h \defeq 2(L\!-\!m)\!\bmat{1 \\ -1}\!. 
\]
Consider an algorithm in lifted form \eqref{eq:sslifted}. We are interested in writing down as many valid inequalities as we can that relate its iterates. To this end, we will consider nonnegative linear combinations of the inequalities \eqref{eq:qij}. The computation is presented in the following corollary.

\begin{cor}\label{lem:cvx_func_ineq}
    Consider a function $f\in \FmL$, and let $y_\star\in\R^{1\times d}$ denote its optimizer, $u_\star=0\in\R^{1\times d}$ the optimal gradient, and $f_\star\in\R$ the optimal function value. Let $y_k,\dots,y_{k-\ell} \in \R^{1\times d}$ be a sequence of iterates, and define $u_{k-i} \defeq \grad f(y_{k-i})$ and $f_{k-i}\defeq f(y_{k-i})$ for $i=0,\dots,\ell$. Using these values, define the augmented vectors $\y_k$, $\u_k$, $\f_k$ as in~\eqref{eq:lifted_iterates}. Finally, define the index set $\mathcal{I} \defeq \{1,2,\ldots,\ell+1,\star\}$ and let $\e_i$ denote the $i\textsuperscript{th}$ unit vector in $\R^{\ell+1}$ with $\e_\star \defeq 0\in\R^{\ell+1}$. Then the inequality
    \begin{equation}\label{eq:cvx_ineq}
      \trace \bmat{\y^t \\ \u^t}^\tp \Pi(\Lambda) \bmat{\y^t \\ \u^t} + \pi(\Lambda)^\tp \f^t \ge 0
    \end{equation}
    holds for all $\Lambda\in\R^{(\ell+2)\times(\ell+2)}$ such that $\Lambda_{ij}\geq 0$, where $\Pi(\Lambda)$ and $\pi(\Lambda)$ are defined in \cref{eq:multipliers}.
\end{cor}
\begin{figure*}[!b]
    \begin{equation}\label{eq:multipliers}
        \Pi(\Lambda) \defeq \sum_{i,j\in \mathcal{I}} \Lambda_{ij} \bmat{ -mL\,(\e_i-\e_j)(\e_i-\e_j)^\tp & (\e_i-\e_j)(m\e_i-L\e_j)^\tp \\[2pt] (m\e_i-L\e_j)(\e_i-\e_j)^\tp & -(\e_i-\e_j)(\e_i-\e_j)^\tp }
        ,\;\;
        \pi(\Lambda) \defeq 2\,(L-m) \sum_{i,j\in \mathcal{I}} \Lambda_{ij}\,(\e_i-\e_j)
    \end{equation}
\end{figure*}

\section{Lyapunov performance certification}

In \cref{prop:quadratic,prop:estimating_sequences}, we used \emph{convergence rate} as a proxy for algorithm performance. For an algorithm $G$ of the form \eqref{eq:ssalg} with fixed point $\xi_\star$, we define convergence rate formally as follows. Assume $w_k=0$ for all $k$ and let
\begin{multline*}
    \rate(G)
    \defeq \\
    \inf\left\{ r > 0 \;\middle|\; \sup_{f\in\FmL}\, \sup_{\xi_0\in\R^{n\times d}}\,\sup_{k\geq 0}\, \frac{\norm{\xi_k - \xi_\star}}{r^k \norm{\xi_0-\xi_\star}}<\infty \right\}.
\end{multline*}
A smaller $\rho_G$ is desirable because it means that the algorithm is guaranteed to converge faster to its fixed point.

Another performance metric of interest is \emph{sensitivity to additive gradient noise}.
Suppose the noise inputs $w = (w_0,w_1,\dots)$ are random, zero-mean, bounded variance, and independent across timesteps (not necessarily identically distributed). Specifically, $\E w_k = 0$, $\E w_k^\tp w_k \preceq \sigma^2 I_d$, and $\E w_i^\tp w_j = 0 $ for all $i\neq j$. We denote the set of all such joint distributions over $w$ as $\mathcal{P}_\sigma$.


For any fixed algorithm $G$, function $f\in\FmL$, initial point $\xi_0$, and noise distribution $w\sim \mathbb{P} \in \mathcal{P}_\sigma$, consider the stochastic iterate sequence $y_0,y_1,\dots$ produced by $G$ and let $y_\star \defeq \argmin_{y\in\R^d} f(y)$ be the unique minimizer of $f$. We define the noise sensitivity to be:
\begin{multline*}\label{def:gamma}
\sensitivity(G,\sigma^2) \defeq \sup_{f\in\FmL}\, \sup_{\xi_0\in\R^{n\times d}}\, \sup_{\mathbb{P}\in \mathcal{P}_\sigma}\\
\limsup_{T\to \infty}  \sqrt{\E_{w\sim \mathbb{P}}\, \frac{1}{T} \sum_{k=0}^{T-1} \normm{ y_k-y_\star }^2}.
\end{multline*}
A smaller sensitivity is desirable because it means that the algorithm is more robust to gradient noise.

Both convergence rate and noise sensitivity can be computed exactly using eigenvalue analysis for the case of quadratic functions \cite{mihailo,optalg}. For smooth strongly convex functions, IQC theory can instead be used to bound the rate and sensitivity \cite{scherer}. 

We will now show an alternative Lyapunov approach, originally proposed in \cite{dissalg,optalg,van2022absolute}. Specifically, we will use a Lyapunov function of the Lur'e--Postnikov form (quadratic in the state plus integral of the nonlinearity). For our lifted system representation, this takes the form
\begin{equation}\label{eq:lyap}
  V(\x,\f) \defeq \trace\bigl(\x^\tp P \x\bigr) + p^\tp Z\f,
\end{equation}
where $P$ and $p$ are parameters that must be optimized ($P$ need not be positive definite), and the matrix $Z$ is defined in~\eqref{eq:Z}. Both convergence rate and noise sensitivity can be verified by finding $P$ and $p$ such that the Lyapunov function \eqref{eq:lyap} satisfies some algebraic conditions.

\begin{lem}\label{lem:rate}
Consider the algorithm dynamics \eqref{eq:ssalg} and let $V$ be of the form \eqref{eq:lyap}. If the iterates of the lifted system \eqref{eq:sslifted} with $w_k=0$ satisfy the conditions
\begin{enumerate}[(i)]
    \item\label{ratei} $V(\x_k,\f_k) \geq \norm{\xi_k-\xi_\star}^2$ and
    \item\label{rateii} $V(\x_{k+1},\f_{k+1}) \leq r^2\,V(\x_k,\f_k)$ for some $r>0$,
\end{enumerate}
then $\rate(G) \leq r$.
\end{lem}
\begin{proof}
Applying \cref{ratei,rateii} for $k \geq \ell$, we obtain
\begin{align*}
    \norm{\xi_k-\xi_\star}^2 &\leq V(\x_k,\f_k) \leq \cdots \leq r^{2k}V(\x_0,\f_0) \\
    &\leq r^{2k} \bigl(C_0 \norm{\xi_0 - \xi_\star}^2 + C_1 \bigr),
\end{align*}
where $C_0$ and $C_1$ are constants that depend on the initialization of the algorithm and the parameters $P,p$. The result follows from applying the definition of $\rate(G)$.
\end{proof}

\begin{lem}\label{lem:sens}
    Consider the algorithm dynamics \eqref{eq:ssalg} and let $V$ be of the form \eqref{eq:lyap}. If the iterates of the lifted system \eqref{eq:sslifted} satisfy the conditions
    \begin{enumerate}[(i)]
        \item\label{sensi} $\E V(\x_k,\f_k) \geq 0$
        \item\label{sensii} $\E V(\x_{k+1},\f_{k+1}) - \E V(\x_k,\f_k) + \E\,\norm{y_k-y_\star}^2 \leq \gamma^2$ for some $\gamma > 0$,
    \end{enumerate}
    then $\sensitivity(G,\sigma^2) \leq \gamma$.
\end{lem}
\begin{proof}
applying \cref{sensi} and averaging \cref{sensii} over $k=0,\dots,T-1$, we obtain
\[
    \frac{1}{T} \E \sum_{t=0}^{T-1} \norm{y_k-y_\star}^2 \leq \frac{1}{T} \E V(\x_0,\f_0) + \gamma^2.
\]
Taking the limit superior as $T\to\infty$ implies that $\gamma$ is an upper bound on the senstivity to gradient noise.
\end{proof}

We will now show how \cref{lem:rate,lem:sens}, together with the interpolation conditions of \cref{sec:interpolation}, can be used to efficiently certify convergence rate and noise sensitivity for iterative algorithm.


\begin{thm}\label{thm:lmi_cvx}
	Consider an algorithm $G$ of the form \eqref{eq:ssalg} with with fixed point $(\xi_\star,y_\star,u_\star)$ applied to a function $f \in \FmL$. Suppose there is additive gradient noise with distribution in $\mathcal{P}_\sigma$. Define the truncation matrices in \eqref{eq:Z}, the augmented state space and projection matrices in \eqref{eq:sslifted}--\eqref{eq:sslifted2}, and the valid inequality matrices in \eqref{eq:multipliers}.
	\begin{enumerate}[a)]
 	\item If there exist $P=P^\tp\in\R^{(n+2\ell)\times (n+2\ell)}$ and $p\in\R^\ell$ and $\Lambda_1,\Lambda_2\geq 0$ (elementwise) and $r>0$ such that
 	\begin{subequations}\label{eq:lmi_cvx_rate}
	\begin{align}
    \squeezemat{-2pt}\hspace{-5mm}
    \bmat{ \A & \B \\ I & 0 \\ \C & \D}^\tp\!
    \bmat{ P & 0 & 0\\ 0 & -r^2 P & 0 \\ 0 & 0 & \Pi(\Lambda_1)}\!
    \bmat{ \A & \B \\ I & 0 \\ \C & \D } &\preceq 0
    \label{eq:lmi_cvx_1} \\
    (Z_+ - r^2 Z)^\tp p + \pi(\Lambda_1) &\leq 0
    \label{eq:lmi_cvx_2} \\
    \hspace{-1cm}\X^\tp \X + \squeezemat{-2pt}
    \bmat{I & 0 \\ \C & \D}^\tp \!
    \bmat{-P & 0 \\ 0 & \Pi(\Lambda_2)}\!
    \bmat{I & 0 \\ \C & \D} &\preceq 0
    \label{eq:lmi_cvx_3} \\
    -Z^\tp p + \pi(\Lambda_2) &\leq 0
    \label{eq:lmi_cvx_4}
	\end{align}
  \end{subequations}
	then $\rate(G) \leq r$.
	\item If there exist $P=P^\tp\in\R^{(n+2\ell)\times(n+2\ell)}$ and $p\in\R^\ell$ and $\Lambda_1,\Lambda_2\geq 0$ (elementwise) such that
	\begin{subequations}\label{eq:lmi_cvx_sensitivity}
	\begin{align}
        \squeezemat{-2pt}\hspace{-1cm}
        \bmat{ \A & \B \\ I & 0 \\ \C & \D}^\tp \!
        \bmat{ P & 0 & 0\\ 0 & -P & 0\\ 0 & 0 & \Pi(\Lambda_1)} \!
        \bmat{ \A & \B \\ I & 0 \\ \C & \D}\! + \Y^\tp \Y &\preceq 0
    \label{eq:lmi_cvx_sensitivity_1} \\
    (Z_+ - Z)^\tp p + \pi(\Lambda_1) &\leq 0
    \label{eq:lmi_cvx_sensitivity_2} \\
    \squeezemat{-2pt}
    \bmat{I & 0 \\ \C & \D}^\tp \!
    \bmat{-P & 0 \\ 0 & \Pi(\Lambda_2)} \!
    \bmat{I & 0 \\ \C & \D} &\preceq 0
    \label{eq:lmi_cvx_sensitivity_3} \\
    -Z^\tp p + \pi(\Lambda_2) &\leq 0
    \label{eq:lmi_cvx_sensitivity_4}
  \end{align}
  \end{subequations}
	then $\sensitivity(G,\sigma^2) \le \sqrt{\sigma^2 d\cdot (\H^\tp P \H)}$.
	\end{enumerate}
\end{thm}

\begin{proof}
  Consider a trajectory $(\xi_k,u_k,y_k,w_k)$ of the dynamics \eqref{eq:ssalg} with $w_k=0$. Multiply \eqref{eq:lmi_cvx_1} and \eqref{eq:lmi_cvx_3} on the right and left by $(\x_k,u_k)\in\R^{n+\ell+1}$ and its transpose, respectively, and take the trace. Also, take the inner product of \eqref{eq:lmi_cvx_2} and \eqref{eq:lmi_cvx_4} with $\f_k$, which is valid because $\f_k$ is elementwise nonnegative. Next, sum the resulting \eqref{eq:lmi_cvx_1}+\eqref{eq:lmi_cvx_2} and \eqref{eq:lmi_cvx_3}+\eqref{eq:lmi_cvx_4}, and the result immediately follows from \cref{lem:rate} and \cref{eq:cvx_ineq}.

  For the second part of the proof, we do not restrict $w_k=0$, and perform similar operations to the inequalities \eqref{eq:lmi_cvx_sensitivity} as in the first part, this time taking expected values, and applying \cref{lem:sens} and \cref{eq:cvx_ineq}.
\end{proof}

\subsection{Efficient numerical solutions}

For fixed $r>0$, the conditions \eqref{eq:lmi_cvx_rate} and \eqref{eq:lmi_cvx_sensitivity} are linear matrix inequalities (LMIs) in the variables $P,p,\Lambda_1,\Lambda_2$. Each LMI has a size that depends on the number of algorithm states $n$ and the lifting dimension $\ell$, both of which are typically small. Critically, the size of the LMIs \emph{does not} depend on $d$ (the domain dimension of $f$), which can be very large in practice.

To find the best bound on $\rate(G)$, one can perform a bisection search on $r$, at each step checking feasibility of \cref{eq:lmi_cvx_rate}. To find the best bound on $\sensitivity(G,\sigma^2)$, we can directly minimize $\H^\tp P \H$ subject to \cref{eq:lmi_cvx_sensitivity}.

\section{Numerical examples}\label{sec:numerical}

In \cref{fig:rate_comparison}, we plot bounds on $\rate(G)$ for smooth strongly convex functions for different algorithms and choices of $L/m$, computed using \cref{thm:lmi_cvx}. For all curves, a lifting dimension $\ell=1$ was enough to get the best results. We tuned \eqref{GD} and \eqref{HB} as in \cref{prop:quadratic}, which recovers the result of \cite[\S4.6]{lessard16} whereby HB tuned in this manner may not be globally convergent in $\FmL$. We tuned \eqref{FG} as in \cref{prop:quadratic}, which yields a tighter rate bound that that of \cref{prop:estimating_sequences} (shown as FG*). Finally, we show the \emph{Triple Momentum Method} \cite{van2017fastest}, which has the fastest known convergence rate for this function class.

\begin{figure}[ht]
\centering
\includegraphics{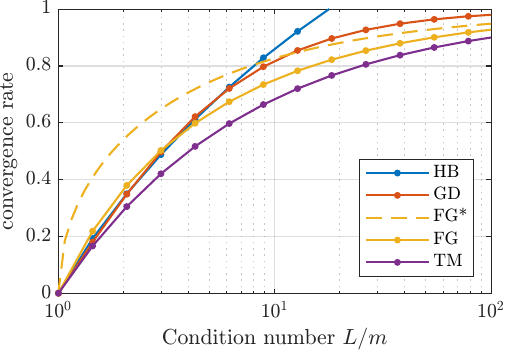}
\caption{Application of \cref{thm:lmi_cvx} to find bounds on $\rate(G)$ for different algorithms applied to smooth strongly convex functions. See the text of \cref{sec:numerical} for details.}
\label{fig:rate_comparison}
\end{figure}

In \cref{fig:perf_tradeoff}, we plot sensitivity to gradient noise versus convergence rate for various algorithms applied to smooth strongly convex functions with $m=1$ and $L=8$, computed using \cref{thm:lmi_cvx}. We used $\ell=1$ for convergence rate and $\ell=6$ for noise sensitivity. Also shown is GD*, which explores other tunings of GD with $0 \leq \alpha \leq \tfrac{2}{L}$ and shows that the regime $\alpha > \frac{2}{L+m}$ is always Pareto-suboptimal. Finally, we show RAM, which is the \emph{Robust Accelerated Method} of \cite{optalg} and achieves the best known trade-off between convergence rate and noise robustness.

\begin{figure}[ht]
\centering
\includegraphics{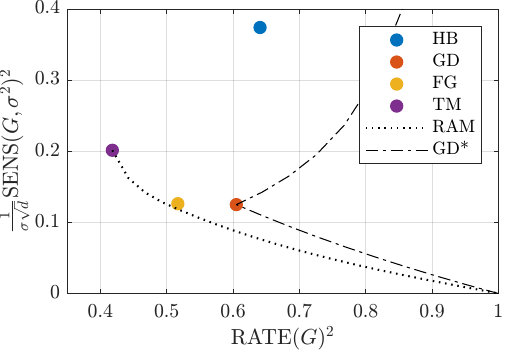}
\caption{Application of \cref{thm:lmi_cvx} to plot the trade-off between sensitivity to additive gradient noise and convergence rate in $\mathcal{F}_{1,8}$. See the text of \cref{sec:numerical} for details.}
\label{fig:perf_tradeoff}
\end{figure}

\section{Concluding remarks}

In this tutorial, we presented a Lyapunov-based approach for algorithm analysis, which is covered in more technical detail in \cite{optalg,dissalg,van2022absolute}.

Showcased in \cref{fig:perf_tradeoff,fig:rate_comparison}, this approach attains numerical results that empirically match those obtained using IQCs \cite{lessard16,scherer}, but does so using a familiar Lur'e--Postnikov Lyapunov function \eqref{eq:lyap}, and with results such as \cref{thm:lmi_cvx} that have straightforward proofs.

\newpage
\bibliographystyle{IEEEtran}
{\footnotesize\bibliography{lyaplift}}

\end{document}